\documentclass{amsart}
\usepackage{amsmath,amssymb, amscd, graphicx}
\usepackage{youngtab}

\usepackage{enumerate}
\usepackage{color}

\newcommand{\revsMM}[1]{#1}

\def\ii{{\sqrt{-1}}}

\def\tX{{\widetilde{X}}}

\def\hPsi{{\widehat{\Psi}}}

\def\hmu{{\widehat{\mu}}}
\def\hphi{{\widehat{\phi}}}

\def\tX{{{\widetilde{X}}}}

\def\tX{\widetilde{X}}
\def\Xrs{{X}}

\def\hS{{\widehat{S}}}

\def\fB{{\mathfrak{B}}}
\def\cK{{\mathcal{K}}}

\def\wt{\mathrm {wt}_\infty}

\def\nuI{{\nu}}
\def\hS{{\widehat{S}}}
\def\hnuI{{\nu^\circ}}
\def\twc{{\widetilde{w}^\circ}}
\def\wc{{{w}^\circ}}

\def\wc{{{w}^\circ}}

\def\CC{{\mathbb C}}
\def\ZZ{{\mathbb Z}}
\def\NN{{\mathbb N}}

\def\RR{{\mathbb R}}

\def\WW{{{\mathcal W}}}

\def\cO{{\mathcal{O}}}
\def\cA{{\mathcal{A}}}

\def\PP{{\mathbb P}}

\def\JJ{{\mathcal J}}

\def\cS{{\mathcal S}}

\def\Sp{{\mathrm{Sp}}}

\def\tw{{{\widetilde{w}}}}

\def\ker{\mathrm {ker }}

\def\phiH#1{{\widehat{\phi}_{#1}}}

\def\hpsig#1{{\widehat\psi_{#1}}}
\def\hPsig#1{{\widehat\Psi_{#1}}}

\def\hmu#1{{\widehat\mu_{#1}}}

\newtheorem{theorem}{Theorem}[section]
\newtheorem{definition}[theorem]{Definition}

\newtheorem{proposition}[theorem]{Proposition}

\newtheorem{remark}[theorem]{Remark}
\newtheorem{lemma}[theorem]{Lemma}

\def\dfrac#1#2{{\displaystyle\frac{#1}{#2}}}

\def\book#1{\rm{#1}, }
\def\paper#1{\textit{#1}, }
\def\jour#1{\rm{#1}, }
\def\yr#1{({\rm{#1}) }}
\def\vol#1{\textbf{#1}}
\def\pages#1{\rm{#1}}

\def\publ#1{\rm{#1}, }
\def\by#1{{\rm{#1}, }}

\begin{document}

\title{Jacobi inversion formulae for a curve in Weierstrass normal form}

\author[J, Komeda]{Jiyro Komeda}
\address[J, Komeda]{Department of Mathematics,
Center for Basic Education and Integrated Learning,
Kanagawa Institute of Technology,
1030 Shimo-Ogino, Atsugi, Kanagawa 243-0292, Japan}
\email{komeda@gen.kanagawa-it.ac.jp}
\author[S. Matsutani]{Shigeki Matsutani}
\address[S. Matsutani]{Industrial Mathematics,
National Institute of Technology, Sasebo College,
1-1 Okishin, Sasebo, Nagasaki, 857-1193, Japan}
\email{smatsu@sasebo.ac.jp}

\date{}

\begin{abstract}
We consider 
a pointed curve $(X,P)$ which is given by the Weierstrass normal form,
$y^r + A_{1}(x) y^{r-1} + A_{2}(x) y^{r-2} +\cdots + A_{r-1}(x) y
 + A_{r}(x)$ where
$x$ is an affine coordinate on $\mathbb{P}^1$,
the point $\infty$ on $X$ is mapped to $x=\infty$, and 
 each $A_j$ is a
polynomial in $x$ of degree $\leq js/r$ for a certain coprime positive
integers $r$ and $s$ ($r<s$)
so that its Weierstrass non-gap sequence at $\infty$ is a numerical 
semigroup.
It is a natural generalization of Weierstrass' equation in the 
Weierstrass elliptic function theory.
We investigate such a curve and
show the Jacobi inversion formulae of the strata of its Jacobian
using the result of Jorgenson \cite{Jo}. 
\end{abstract}




\maketitle

\begin{center}
Dedicated for Emma Previato's 65th birthday.
\end{center}

\section{Introduction}\label{introduction}

The Weierstrass $\sigma$ function
is defined for an elliptic curve of Weierstrass' equation
$y^2 = 4 x^3 - g_2 x - g_3$ in Weierstrass' elliptic function theory \cite{WW};
$\Bigr(\wp(u)=-\dfrac{d^2}{du^2}\ln\sigma(u), $ 
$\dfrac{d\wp(u)}{du}\Bigr)$ 
is identical to a point $(x,y)$ of the curve. 
The $\sigma$ function is related to Jacobi's $\theta$ function. 
As the $\theta$ function was generalized by Riemann for any Abelian 
variety,
its equivalent function Al was defined for any hyperelliptic Jacobian by 
Weierstrass \cite{W54} and was refined by 
Klein \cite{Klein1, Klein2} as a generalization of the 
elliptic $\sigma$ function such that
 it satisfied a modular invariance
under the action of $\Sp(2g, \mathbb{Z})$ (up to a root of unity).

Recently the studies on the $\sigma$ functions in the XIXth century
have been reevaluated and reconsidered.
Grant and \^Onishi gave their modern perspective and
showed precise structures of hyperelliptic Jacobians
from a viewpoint of number theory using the $\sigma$ functions
\cite{Gr,O}, whereas
Buchstaber,  Enolski\u{\i} and  Le\u{\i}kin \cite{BEL}
and 
Eilbeck, Enolski\u{\i} and  Le\u{\i}kin \cite{EEL} investigated the
hyperelliptic $\sigma$ functions
for their applications to the integrable system.
One of these authors also applied Baker's results to 
the dynamics of loops in a plane \cite{Pr93,MP16}
associated with the modified Korteweg-de Vries hierarchy \cite{M02,M10,MP16}.
Further Buchstaber,  Le\u{\i}kin and Enolski\u{\i}  \cite{BLE}
and Eilbeck, Enolski\u{\i} and  Le\u{\i}kin \cite{EEL} 
generalized the $\sigma$ function to more general plane curves.
Nakayashiki connected these $\sigma$ functions with 
the $\theta$ functions in Fay's study \cite{F} and   
the $\tau$ functions in the Sato universal Grassmannian theory \cite{N10,N16}.

Weierstrass' elliptic function theory provides the concrete and
explicit descriptions of the geometrical, algebraic and analytic
properties of elliptic curves and their related functions \cite{WW},
and thus it has strong effects on various fields in mathematics, physics and
technology. We have reconstructed the theory of Abelian functions of
curves with higher genus to give the vantage point like Weierstrass's
elliptic function theory [MK, MP1-5, KMP2-4].


In the approach of the $\sigma$ function,
 the pointed curve $(X,P)$ is crucial,
since the relevant objects are written in terms of 
$H^0(X, \cO(*P))$. 
The representation of the affine curve 
$X\setminus P$ is therefore also relevant,
and so is the Weierstrass semigroup (W-semigroup) at $P$.
It is critical to find the proper basis of $H^0(X, \cO(*P))$
to connect the (transcendental) 
$\theta$ ($\sigma$) function with the (algebraic) functions 
and differentials of the curve;
in  Mumford's investigation, the basis corresponds to $U$ function
of the Mumford triplet, $UVW$ \cite{Mu}, which is identical to
a square of Weierstrass $al$ function \cite{W54}.
The connection means the Abel-Jacobi map and 
its inverse correspondence, or the Jacobi inversion formula.

Using the connection, we represent the affine curve $X\setminus P$
in terms of the $\sigma$ ($\theta$) functions and find
these explicit descriptions of the geometrical, algebraic and analytic
properties of the curve and their related functions.
These properties are given as differential equations and related to
the integrable system as in \cite{B03,Mu,Pr85,Pr86,EEMOP08}.
Since the 
additional structure of the hyperelliptic curves \cite{EEMOP07}
is closely related
to Toda lattice equations and classical Poncelet's problem, 
the additional structures were also revealed as dynamical systems 
\cite{KMP12}.

Using the $\sigma$ functions, we have the Jacobi inversion formulae
of the Jacobians, 
e.g., 
for hyperelliptic curves \cite{W54,B97},
for cyclic $(n,s)$ curves (super-elliptic curves) including
the strata of the Jacobians \cite{MP08,MP14}, for curves with
``telescopic'' Weierstrass semigroups \cite{A16}
and cyclic trigonal curves \cite{MK, KMP13,KMP16,KMP18}. 
In this paper, we consider the Jacobi inversion 
problem for a general Riemann surface.
The main effort in these studies is directed toward explicitness. 

Since it is known that every compact Riemann surface is birationally-equivalent
to a curve given by the Weierstrass normal form \cite{WN,B97,kato},
in this paper, we investigate the Jacobi inversion formula of such a curve,
which is called Weierstrass curve; its definition is given in 
Proposition \ref{prop:WNC}.
We reevaluate Jorgenson's results for the $\theta$ divisor of
a Jacobian by restricting ourselves to the Jacobian of a Weierstrass curve.
Then we show explicit Jacobi inversion formulae 
even for the strata of the 
$\theta$ divisor 
in Theorem \ref{thm:JIF} as our main theorem of this paper.

This approach enables us to investigate the properties of 
image of the Abelian maps of a compact Riemann surface more precisely.
In the case of the hyperelliptic Jacobian,
the sine-Gordon equation gives the relation among the meromorphic
functions on the Jacobian, which is proved by the Jacobi inversion 
formula.
Further the Jacobi inversion formula of 
the stratum of the hyperelliptic Jacobian generalizes the sine-Gordon
equation to differential relations in 
the strata \cite{M05}, 
Very recently, Ayano and Buchstaber found  
novel differential equations that characterizes the stratum of
the hyperelliptic Jacobian of genus three \cite{AB}. 
Accordingly, it is
expected that the results in this paper might
also lead a generalization of the integrable system and 
bring out the geometric and algebraic data of the strata of the Jacobians.

We use the word ``curve'' for a compact Riemann surface: on occasion, we use a
singular representation of the curve; since there is a unique smooth curve
with the same field of meromorphic function, this should not cause confusion.
We let the non-negative integer denoted by $\NN_0$ 
and the positive integer by $\NN$.

The contents  of this paper are organized as follows:
in Section 2, we collect definitions, properties,
and examples of Weierstrass curves and Weierstrass semigroups.
Section 3 provides the Abelian map (integral) and  the $\theta$ functions
 as the transcendental properties of the Abelian map  (integral) image.
We also mentioned Jorgenson's result there.
In Section 4 we show the properties of holomorphic one forms
of Weierstrass curve $(X, P)$ in terms of 
the proper basis of $H^0(X, \cO(*P))$ as
the algebraic property of the curve $X$.
In Section 5 we applied these data to the result of Jorgenson
and show our main theorem, the explicit representations of 
Jacobi inversion formulae
of a general Weierstrass curve in Theorem \ref{thm:JIF}.

\noindent
{\bf{
Acknowledgments:}}
The second named author 
thanks Professor Victor Enolskii for the notice on the paper of Jorgenson 
\cite{Jo}, 
though he said that it was learned from Professor Emma Previato.
The authors
 were supported by
the Grant-in-Aid for Scientific 
Research (C) of Japan Society for the Promotion
of Science, 
Grant No. 18K04830 and Grant No. 16K05187 respectively.
Finally, the authors would like to thank 
Professor Emma Previato for her extraordinary generosity
and long collaboration, and dedicate this paper to her.

\section{Weierstrass normal form and Weierstrass curve}\label{weierstrass}
We set up the notations for  Weierstrass semigroups and Weierstrass normal
form, and recall the results we use.

\subsection{Numerical semigroups and Weierstrass semigroups}
\label{weierstrassN}
A pointed curve is a pair $(X,P)$, with $P$ a point of a curve $X$;
the Weierstrass semigroup for $X$ at $P$, which we denote by $H(X,P)$, 
 is the complement 
of the Weierstrass gap sequence $L$, namely  the set of natural
numbers 
$\{\ell_0 < \ell_1 < \cdots < \ell_{g-1} \}$ such that 
$H^0(\cK_X-\ell_iP)\ne H^0(\cK_X-(\ell_i-1)P)$, for
$\cK_X$ a  
representative of the canonical divisor (we identify
divisors with the corresponding sheaves).
By the Riemann-Roch theorem, $H(X,P)$
 is a numerical semigroup.
In general, a numerical semigroup $H$ has 
a unique (finite) minimal set of generators,  $M=M(H)$ and
the finite cardinality $g$  of $L(M)=\NN_0\backslash H$;
$g$ is the genus of $H$ or $L$.
We let $a_{\min}(H)$ be the smallest positive integer of $M(H)$.
For example,
\begin{gather*}
\begin{split}
&L(M)=\{1,2,3,4,6,8,9,11,13,16,18,23\}, \
 a_{\min}=5,\ 
 \mbox{for} \ M=\langle 5,7\rangle,\\
&L(M)=\{1,2,3,4,6,8,9,13\}, \  a_{\min}=5,\ 
\mbox{for} \ M=\langle 5,7,11\rangle,
\mbox{ and } \\
&L(M)=\{1, 2, 4, 5\}, \  a_{\min}=3, \ \mbox{for} \ 
M=\langle 3,7,8\rangle.\\
\end{split}
\end{gather*}
The Schubert index of the set $L(M(H))$ is
\begin{equation}
\alpha(H) :=\{\alpha_0(H), \alpha_1(H), \ldots, \alpha_{g-1}(H)\},
\label{eq:alphaL}
\end{equation}
where 
$\alpha_i(H) := \ell_i - i -1$ \cite{EH}.
By letting the row lengths be
$\Lambda_i = \alpha_{g-i}+1$, $i\le g$,
we have the Young diagram of the semigroup,
$\Lambda:
=(\Lambda_1, ..., \Lambda_g)$.
If for a numerical semigroup $H$,  there exists a curve whose
Weierstrass non-gap sequence is identical to $H$, we call the semigroup
$H$ {\it{Weierstrass}}.
It is known that every numerical semigroup is not Weierstrass.
A Weierstrass semigroup is called symmetric when 
$2g-1$ occurs in the gap sequence.
It implies that
$H(X,P)$ is symmetric if and only if its Young diagram is symmetric,
in the sense of being invariant under reflection across the main diagonal.

\subsection{Weierstrass normal form \revsMM{and Weierstrass curve}}

We now review the ``Weierstrass normal form'',
which is a
generalization of Weierstrass' equation for elliptic curves \cite{WW}.
Baker \cite[Ch. V, \S\S 60-79]{B97} 
gives a complete review, proof and examples of the theory,
though he calls it ``Weierstrass canonical form''.  
Here we refer to Kato \cite{kato}, who 
also produces this representation, with proof.

Let $m=a_{\mathrm{min}}(H(X,P))$ and let $n$ be the least integer in $H(X,P)$ 
which is prime to $m$.  By denoting $P= \infty$,
$X$ can be viewed as an $m:1$  cover of $\mathbb{P}^1$
via an equation of the following form:
$f(x,y)=y^m+A_1(x)y^{m-1}+\cdots+A_{m-1}(x)y+A_m(x)=0$, where 
$x$ is an affine coordinate on $\mathbb{P}^1$,
the point $\infty$ on $X$ is mapped to $x=\infty$, and 
 each $A_j$ is a
polynomial in $x$ of degree $\leq jn/m$, with equality being attained only for
$j=m$.  The algebraic curve Spec $\CC[x,y]/f(x,y)$, is, in general,
singular and we denote by $X$ its unique normalization.
Since \cite{kato} is only available in Japanese, and since the
meromorphic functions constructed in his proof will be used in our examples
 in \ref{rmk:WNF} below,
we reproduce his proof in sketch:
it shows that the
affine ring of the curve $X\backslash\infty$ can be generated by functions that have
poles  at $\infty$ corresponding to the $g$ non-gaps.

\begin{proposition} {\rm{\cite{kato}}} \label{prop:WNC}
For a pointed curve $(X,\infty )$ 
with Weierstrass semigroup $H(X,\infty)$
for which  $a_{\mathrm{min}}(H(X,\infty))=m$, 
we let
$m_i :=\min\{h \in H(X,\infty)\setminus \{0\}\ 
|\ i \equiv h \ \mbox{ modulo } m\}$,
$i =0, 1, 2, \ldots, m-1$,
$m_0 = m$ and  $n=\min\{m_j \ | \ (m, j) = 1\}$.
$(X,\infty)$
 is defined by an irreducible equation,
\begin{equation}
f(x,y)= 0,
\label{eq:WNF1a}
\end{equation}
for a polynomial $f\in \CC[x,y_n]$ of type,
\begin{equation}
f(x,y):=y^m + A_{1}(x) y^{m-1}
+ A_{2}(x) y^{m-2}
+\cdots + A_{m-1}(x) y
 + A_{m}(x),
\label{eq:WNF1b}
\end{equation}
where the $A_i(x)$'s are polynomials in $x$,
$$
A_i = \sum_{j=0}^{\lfloor i n/m\rfloor} \lambda_{i, j} x^j,
$$
and $\lambda_{i,j} \in \CC$, $\lambda_{m,m}=1$.
\end{proposition}

We call the pointed curve given in Proposition
 \ref{prop:WNC} {\it{Weierstrass curve}} in this paper.

\begin{proof}
We let
$I_m :=\{m_1, m_2, \ldots, m_{m-1}\}\setminus \{m_{i_0}\}$,
$=:\{m_{i_1}, m_{i_2}, \ldots, m_{i_{m-2}}\}$,
where $i_0$ is such that $n = m_{i_0}$.
Let $y_{m_i}$ be a meromorphic function on $X$ whose only pole is
 $\infty$ with order $m_i$, taking 
 $x=y_{m}$ and $y=y_{n}$.
From the definition of $X$,
we have, as  $\CC$-vector spaces,
$$
H^0(X,\cO_X(*\infty)) =\CC \oplus \sum_{i=0}^{m-1}\sum_{j=0} \CC x^j y_{m_i}.
$$
Thus for every $i_j\in I_m$ $(j=1,2,\ldots,m-2)$,
we obtain the following equations
\begin{gather}
\begin{split}
\left\{
\begin{array}{rcl}
y_{n} y_{m_{i_1}} &=& A_{1,0} +
 A_{1,1} y_{m_{i_1}}+ \cdots +
 A_{1,m-2} y_{m_{i_{m-2}}}+ 
 A_{1,m-1} y_{n}, \\ 
y_{n} y_{m_{i_2}} &=& A_{2,0} +
 A_{2,1} y_{m_{i_1}}+ \cdots +
 A_{2,m-2} y_{m_{i_{m-2}}}+ 
 A_{2,m-1} y_{n}, \\ 
& & \cdots\\
y_{n} y_{m_{i_{m-2}}} &=& A_{{m-2},0} +
 A_{m-2,1} y_{m_{i_1}}+ \cdots +
 A_{m-2,m-2} y_{m_{i_{m-2}}}+ 
 A_{{m-2},m-1} y_{n}, \\ 
\end{array}\right.
\label{eq:WNF2}
\end{split}
\end{gather}
\begin{gather}
y_{n}^2 = A_{{m-1},0} +
 A_{{m-1},1} y_{m_{i_1}}+ \cdots +
 A_{{m-1},m-2} y_{i_{m-2}}+ 
 A_{{m-1},m-1} y_{n},
\label{eq:WNF3}
\end{gather}
where $A_{i,j} \in \CC[x]$.

When $m = 2$, (\ref{eq:WNF1a}) equals  (\ref{eq:WNF3}).
We assume that $m>2$ and then (\ref{eq:WNF2}) is reduced to
\begin{gather}
\begin{split}
\begin{pmatrix}
A_{1,1} - y_n & A_{1,2} & \cdots & A_{1,m-2}\\
A_{2,1} & A_{2,2} - y_n& \cdots & A_{2,m-2}\\
\vdots & \vdots & \ddots & \vdots \\
A_{m-2,1} & A_{m-2,2} & \cdots & A_{m-2,m-2} - y_n\\
\end{pmatrix}
\begin{pmatrix}
y_{m_{i_1}}\\
y_{m_{i_2}}\\
\vdots\\
y_{m_{i_{m-2}}}\\
\end{pmatrix}\qquad&\\
=-
\begin{pmatrix}
A_{1,0} + A_{1, m-1} y_n\\
A_{2,0} + A_{2, m-1} y_n\\
\vdots\\
A_{m-2,0} + A_{m-2, m-1} y_n\\
\end{pmatrix}.&
\label{eq:WNF4}
\end{split}
\end{gather}

One can check that the determinant of the 
 matrix on the left-hand side of (\ref{eq:WNF4}) is not
equal to zero by computing the order of pole at $\infty$ of 
the monomials 
$B_{i} y_n^{m-2-i}$ in the expression,
\begin{gather*}
\begin{split}
P(x, y_n)&:=\left|
\begin{matrix}
A_{1,1} - y_n & A_{1,2} & \cdots & A_{1,m-2}\\
A_{2,1} & A_{2,2} - y_n& \cdots & A_{2,m-2}\\
\vdots & \vdots & \ddots & \vdots \\
A_{m-2,1} & A_{m-2,2} & \cdots & A_{m-2,m-2}-y_n\\
\end{matrix}
\right|\\
&=y_n^{m-2} + B_1 y_n^{m-3} +\cdots
+ B_{m-3} y_n + B_{m-4},
\end{split}
\end{gather*}
which is 
$n (m -2 -i) + m \cdot\deg B_i$.
 The fact that $(m,n)=1$
shows that 
$n (m -2 -i) + m \cdot\deg B_i\neq n (m -2 -j) + m \cdot\deg B_j$
for $i\neq j$.

Hence by solving  equation (\ref{eq:WNF4})
we have
\begin{equation}
y_{j_i} = \frac{Q_{i}(x, y_n)}{P(x, y_n)},
\label{eq:WNF5}
\end{equation}
where $j_i \in I_m$,
 $Q_{i}(x, y_n) \in \CC[x,y_n]$ and a polynomial of order 
at most $m-2$ in $y_n$. 
Note that the equations (\ref{eq:WNF5})  are not independent in general
but in any cases the function field  of the  curve   can be generated by
these $y_{j_i}$'s, and its affine ring can be given by
$\CC[x,y_n, y_{a_3}, \ldots, y_{a_\ell}]$
for $j_i \in M_g$, where
$M_g :=\{ a_1, a_2, \ldots, a_\ell\} \subset \NN^{\ell}$ 
with $(a_{k'}, a_k) = 1$ for $k'\neq k$, $a_1=a_{\min}=m, a_2=n$, 
is a minimal set of
generators for $H(X,\infty)$.

By putting (\ref{eq:WNF5}) into (\ref{eq:WNF3}),
we have  (\ref{eq:WNF1b})
since (\ref{eq:WNF1b}) is irreducible. 
\end{proof}

\begin{remark}
{\rm{
Since every compact Riemann surface of genus $g$ has a Weierstrass point whose
Weierstrass gap sequence with genus $g$ \cite{ACGH},
it is characterized by the behavior of the meromorphic functions around 
the point and thus there is a Weierstrass curve which is 
birationally equivalent
to the compact Riemann surface.
The Weierstrass curve
 admits a local $\mathbb{Z}/m\mathbb{Z}$-action at $\infty$,
 in the following sense.
We assume that a curve in this article is a Weierstrass curve which is
given by (\ref{eq:WNF1b}).
We consider the generator $M_g$ of the Weierstrass semigroup $H=H(X,\infty)$
in the proof.
For a polynomial ring $\CC[Z]:=\CC[Z_1, Z_2, \cdots, Z_\ell]$ and a ring
homomorphism,
$$
\varphi :\CC[Z] \to \CC[ t^{a_1}, t^{a_2}, \cdots, t^{a_\ell}], \quad
(\varphi(Z_i) = t^{a_i}, \quad a_i \in M_g),
$$
we consider a monomial ring $B_H:=\CC[Z]/\ker \varphi$. 
The action is defined by sending $Z_i$ to $\zeta_m^{a_i}Z_i,$ 
where $\zeta_m$  is a primitive $m$-th root of unity. 
Sending $Z_1$ to $1/x$ and $Z_{i}$ to $1/y_{a_i}$, 
the monomial ring $B_H$
determines the structure of gap sequence \cite{He, Pi}.

There are natural projections, 
$$
\varpi_x:X \to \PP, \quad
\varpi_y:X \to \PP,
$$
to obtain $x$-coordinate and $y$-coordinate such that
$\varpi_x(\infty) = \infty$ and 
$\varpi_y(\infty) = \infty$.
}}
\end{remark}

\subsection{Examples: Pentagonal, Trigonal  non-cyclic}\label{rmk:WNF}
We clarify that
 the words ``trigonal'' and ``pentagonal'' are used here only as an indication
 of the fact that the pointed curve $(X,\infty )$ has Weierstrass semigroup of type
 three, five, respectively; a different convention requires that $k$-gonal
 curves not be $j$-gonal for $j < k$ (as in trigonal curves, which by
 definition are not hyperelliptic); of course this cannot be guaranteed in our
 examples. Baker \cite[Ch. V, \S 70]{B97} also points this out.    
\rm{{
\begin{enumerate}
\item[I.] Nonsingular affine equation (\ref{eq:WNF1b}):
$y^5 +A_{4}y^4
 +A_{3}y^3 
 +A_{2}y^2 
 +A_{1}y 
 +A_{0}=0$, 
$n=7$ case:
(\ref{eq:WNF4}) corresponds to
$$
\begin{pmatrix}
-y & 0 & 1\\
1  & -y & 0\\
-A_3  & -A_2 & -A_4 - y\\
\end{pmatrix}
\begin{pmatrix}
y^3 \\ y^2\\ y^4
\end{pmatrix}
=\begin{pmatrix}
0 \\ 0 \\ A_1 y + A_0
\end{pmatrix},
$$
(\ref{eq:WNF5}) becomes
$$
y^{2+i}=-\frac{(A_1 y + A_0)y^i}{
y^3+A_4 y^2 + A_3 y + A_2}.
$$
The curve has 5-semigroup at $\infty$ but is not necessarily cyclic. 
\item[II.] Singular affine equation (\ref{eq:WNF1b}):
$y^5 = k_2(x)^2 k_3(x)$, where
$k_2(x) = (x-b_1)(x-b_2)$,
$k_3(x) = (x-b_3)(x-b_4)(x-b_5)$
for pairwise distinct $b_i\in \CC$:
(\ref{eq:WNF4}) corresponds to
$$
\begin{pmatrix}
-y & 1 & 0\\
0  & -y & 0\\
0 & 0 & - y\\
\end{pmatrix}
\begin{pmatrix}
w \\ y w\\y^2
\end{pmatrix}
=\begin{pmatrix}
0 \\ -k_2 k_3 \\ - k_2 w 
\end{pmatrix}.
$$
The affine ring is
$H^0(\cO_X(*\infty)) =\CC[x, y, w]/(y^3-k_2w, w^2-k_3y, y^2 w - k_2 k_3)$.
Here (\ref{eq:WNF5}) is reduced to
\begin{gather*}
w = \frac{k_2 k_3}{y^2},\quad
yw = \frac{k_2 k_3}{y},\quad
y^2 = \frac{k_2 w}{y}.
\end{gather*}
This is a pentagonal cyclic curve $(X,\infty )$ with $H(X,\infty )=\langle
5,7,11\rangle$.

\item[III.] Singular affine equation (\ref{eq:WNF1b}):
$y^3 + a_1 k_2(x) y^2 
+ a_2\tilde k_2(x) k_2(x) y +k_2(x)^2 k_3(x)=0$, 
where $k_2(x) = (x-b_1)(x-b_2)$,
$k_3(x) = (x-b_3)(x-b_4)(x-b_5)$
$\tilde k_2(x) = (x-b_6)(x-b_7)$,
for pairwise distinct $b_i\in \CC$ and $a_j$ generic constants.
Here (\ref{eq:WNF3}) and
(\ref{eq:WNF4}) correspond to
$$
y^2 +a_1k_2 y + k_2 a_2\tilde k_2 + k_2 w =0, \quad
y w = k_2 k_3.
$$
Multiplying the first equation by $y$ 
and using the second equation gives the curve's 
equation.  
Multiplying  that by $w^2$  gives 
$$
w^3 + a_2\tilde k_2 w^2 + a_1k_2 k_3 w + k_2 k_3^2
=0,
$$
which is reduced to the less order relation with respect to $w$,
$$
w^2 + a_2\tilde k_2 w + a_1k_2 k_3  + k_3y =0.
$$
This curve is trigonal with $H(X,\infty )=\langle 3,7,8\rangle$ but 
not necessarily cyclic.
\end{enumerate}

}}

\subsection{Weierstrass non-gap sequence}

Let the commutative ring $R$ be that of 
the affine part of $(X, \infty)$, $R:=H^0(X, \cO(*\infty))$,
which is also obtained by several 
normalizations of 
$\CC[x,y]/(f(x,y))$ of (\ref{eq:WNF1b}) as a normal ring.

\begin{proposition}
The commutative ring $R$ has the basis $S_R:=\{\phi_i\}_i\subset R$
as a set of monomials of $R$ 
such that
$$
  R= \bigoplus_{i=0} \CC\phi_i 
$$
and $\wt \phi_i < \wt \phi_j $ for $i < j$,
where $\wt : R \to \ZZ$ is the order of the singularity at $\infty$.
\end{proposition}

The examples are following.

\hskip -1.0cm
{\tiny{
\hskip -2cm
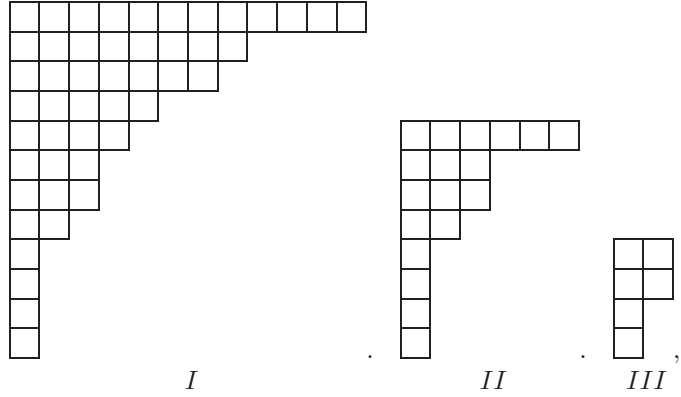
\begin{table}[htb]
	\caption{The $\phi$'s of Examples \S \ref{rmk:WNF} }
  \begin{tabular}{|r|ccccccccccccccc}
\hline
 &0&1&2&3&4&5&6&7&8&9&10&11&12&13&14\\
\hline
 I&$1$&-&-&-&-&$x$&-&$y$&-&-&$x^2$&-&$xy$&-&$y^2$\\
II&$1$&-&-&-&-&$x$&-&$y$&-&-&$x^2$&$w$&$xy$&-& $y^2$\\
 III&$1$&-&-&$x$&-&-&$x^2$&$y$&$w$&$x^3$&$xy$&$xw$&$x^4$&$x^2y$&$y^2$\\
\hline
  \end{tabular}
\\
\hskip 2cm
  \begin{tabular}{cccccccccc|}
\hline
 15&16&17&18&19&20&21&22&23&24\\
\hline
$x^3$&-&$x^2y$&-&$xy^2$&$x^4$&$y^3$&$x^3y$&-&$x^2y^2$\\
$x^3$&$xw$&$x^2y$&$wy$&$xy^2$&$x^4$&$y^3$&$x^3y$&$xyw$&$x^2y^2$\\
$yw$&$w^2$&$xy^2$&$xyw$&$xw^2$&$x^2y^2$&
  $x^2yw$&$x^2w^2$&$x^3y$&$x^3w^2$\\
\hline
  \end{tabular}
\end{table}
}}

$$
\begin{matrix}
\yng(12,8,7,5,4,3,3,2,1,1,1,1).&
\yng(6,3,3,2,1,1,1,1).&
\yng(2,2,1,1),\\
I & II&III\\
\end{matrix}
$$

\section{Abelian integrals and $\theta$ functions}\label{notation}
\subsection{Abelian integrals and $\theta$ functions}

Let us consider a compact Riemann surface $X$ of genus $g$
and its Jacobian $\JJ^\circ(X):={\CC}^g/\Gamma^\circ$ where
$\Gamma^\circ:={\ZZ}^g +\tau {\ZZ}^g$;
$\kappa : \CC^g \to \JJ^\circ(X)$.
Let $\tX$ be an Abelian covering of $X$ ($\varpi: \tX \to X$).
Since the covering space $\tX$ is constructed by a quotient space
of path space (contour of integral) fixing a point $P\in X$,
for $\gamma_{P',P} \in \tX$ whose ending points are $P'\in X$ and $P$,
$\varpi \gamma_{P',P}$ is equal to $P'$, and 
we consider an embedding $X$ into $\tX$ by a map $\iota: X \to \tX$
such that $\varpi\circ \iota = id$.
We fix $\iota$.
For $\gamma_{P',P}\in \tX $, we define the Abelian integral $\twc$
and the Abel-Jacobi map $\wc$,
$$
\twc: \cS^k \tX \to \CC^g, \quad 
\twc(\gamma_{P_1,P}, \cdots, \gamma_{P_k,P}) =
 \sum_{i=1}^k \twc(\gamma_{P_i,P}) =
 \sum_{i=1}^k \int_{\gamma_{P_i,P}} \hnuI, \quad
$$
$$
\wc:=\kappa \circ\twc \circ\iota: \cS^k X \to \JJ^\circ(X),
$$
where $\cS^k \tX$ and $\cS^k X$ are
$k$-symmetric products of $\tX$ and $X$ respectively,
and $\hnuI$ is the column vector of the normalized holomorphic
one forms $\hnuI_i$ $(i=1,2,\ldots,g)$ of $X$;
for the standard basis
$(\alpha_i, \beta_i)_{i=1,\dots,g}$
of Homology of $X$ such that 
$\langle\alpha_i,\alpha_j\rangle=\langle\beta_i,\beta_j\rangle = 0$,
$\langle\alpha_i,\beta_j\rangle = \delta_{ij}$, we have
$$
         \int_{\alpha_i}\hnuI_j = \delta_{ij}, \quad
         \int_{\beta_i}\hnuI_j = \tau_{ij}.
$$
The Abel theorem shows $\kappa \circ\twc = \wc \circ\varpi$.
We fix the point $P \in X$ as a marked point $\infty$ 
of $(X,\infty)$.

The map $\wc$ embeds $X$ into $\JJ^\circ(X)$
and generalizes to a map from the space of divisors of $X$
into $\JJ^\circ(X)$ as 
$\wc(\sum_i n_i P_i):=\sum_in_i\wc(P_i)$, $P_i\in X$,
$n_i\in\ZZ$. Similarly we define $\twc(\iota D)$ for 
a divisor $D$ of $X$.

The Riemann $\theta$ function, analytic in both variables $z$ and $\tau$, is
defined by
$$
\theta(z,\tau ) =
\sum_{n \in \ZZ^{g}} 
\exp\left( 2\pi\ii ({}^t n z + \frac{1}{2}{}^t n \tau n)\right).
$$ 
 The zero-divisor of $\theta$ modulo $\Gamma^\circ$ 
is denoted by  $\Theta:=\kappa\, \mathrm{div}(\theta) \subset \JJ^\circ(X)$.

The $\theta$ function with characteristic $\delta', \delta''\in\RR^{g}$
is defined as:
\begin{equation}
\theta \left[\begin{matrix}\delta'\\ \delta''\end{matrix}
\right] (z, \tau )
   =
   \sum_{n \in \ZZ^g} \exp \big[\pi \sqrt{-1}\big\{
    \ ^t\negthinspace (n+\delta')
      \tau(n+\delta')
   + 2\ {}^t\negthinspace (n+\delta')
      (z+\delta'')\big\}\big].
\label{eq:2.1}
\end{equation}

If $\delta = (\delta',\delta'')\in 
  \{0, \frac{1}{2}\}^{2g}$, then $\theta
\left[\delta\right]\left(z,\tau\right):=\theta
\left[^{\delta'}_{\delta''}\right]\left(z,\tau\right)$ has definite
parity in $z$, $\theta \left[\delta\right]
\left(-z,\tau\right)=e(\delta) \theta \left[\delta\right]
\left(z,\tau\right)$, where $e(\delta):=e^{4\pi \imath{\delta'}^t
\delta''}$. There are $2^{2g}$ different characteristics of definite
parity.

In the paper \cite{KMP16}, we showed that the characteristics are defined 
for a Weierstrass curve by shifting the Abelian integrals and 
the Riemann constant as in Proposition \ref{prop:SRC}.
We recall the basic fact
by Lewittes \cite{Lew}:

\begin{proposition} \label{prop:lew}
{\rm{(Lewittes \cite{Lew})}}

\begin{enumerate}

\item 
Using the Riemann constant $\xi \in \CC^g$,
we have the relation between the $\theta$ divisor 
$\Theta:=$div$(\theta)$ and the standard $\theta$ divisor
$\wc(\cS^{g-1} X)$,
$$
  \Theta = \wc(\cS^{g-1} X) +\xi \quad \mbox{modulo}\quad \Gamma^\circ,
$$
i.e., for $P_i \in X$,
$
\theta(\twc\circ\iota(P_1, \ldots, P_{g-1})+\xi)=0.
$

\item 
The canonical divisor $\cK_X$ of $X$ and the Riemann constant
 $\xi$ have the relation,
$$
\wc(\cK_X) + 2\xi=0\quad \mbox{modulo} \quad\Gamma^\circ.
$$
\end{enumerate}
\end{proposition}

Further in \cite{Jo}, Jorgenson found a crucial relation by
considering the analytic torsion on the Jacobian:

\begin{proposition}\label{prop:Jo}
{\rm{(Jorgenson) \cite[Theorem 1]{Jo}}}
For $P_1, P_2, \ldots, P_{\ell} \in X$, $\ell=g-1$, and general complex numbers
$a_i$ and $b_i$ $(i=1,2, \ldots, g)$, the following holds:
$$
\frac{
\left|\begin{matrix}
\hnuI_1(P_1) & \hnuI_2(P_1) & \cdots &\hnuI_g(P_1)\\
\hnuI_1(P_2) & \hnuI_2(P_2) & \cdots &\hnuI_g(P_2)\\
\vdots & \vdots & \ddots & \vdots \\
\hnuI_1(P_{\ell}) & \hnuI_2(P_{\ell}) & \cdots &\hnuI_g(P_{\ell})\\
a_1 & a_2 & \cdots & a_g \\
\end{matrix}
\right|
}{
\left|\begin{matrix}
\hnuI_1(P_1) & \hnuI_2(P_1) & \cdots &\hnuI_g(P_1)\\
\hnuI_1(P_2) & \hnuI_2(P_2) & \cdots &\hnuI_g(P_2)\\
\vdots & \vdots & \ddots & \vdots \\
\hnuI_1(P_{\ell}) & \hnuI_2(P_{\ell}) & \cdots &\hnuI_g(P_{\ell})\\
b_1 & b_2 & \cdots & b_g \\
\end{matrix}
\right|
}
=
\frac{
\displaystyle{
\sum_{i=1}^g a_i 
\frac{\partial}{\partial z_i }
	\theta(\twc\circ\iota(P_1, P_2, \cdots, P_{\ell})+\xi)
}
}{
\displaystyle{
\sum_{i=1}^g b_i
\frac{\partial}{\partial z_i }
	\theta(\twc\circ\iota(P_1, P_2, \cdots, P_{\ell})+\xi)
}
}.
$$
\end{proposition}

In this paper, we investigate this formula for a Weierstrass curve.

\section{Holomorphic one forms of a Weierstrass curve}

In this section we show several basic properties 
of a Weierstrass curve $(X,\infty)$ with 
 Weierstrass semigroup $H(X,\infty)$, including non-symmetric ones. 
We consider the sheaf of the holomorphic one-from  $\cA_X$ and 
its sections $H^0(X, \cA_X(*\infty))$.

The normalized holomorphic one-forms 
$\{\hnuI_i\ |\ i = 1, \ldots, g\}$ of the Weierstrass curve $(X,P)$
hold the following relation:

\begin{lemma}\label{lem:nuIc}
There are $h \in R$, a finite positive
integer $N$ and 
a subset 
$\hS_h:=\{\phi_{\ell_i}\in S_R\ | \ i= 1, \ldots, N\}$
satisfying
$$
 \hnuI_i = \frac{\sum_{j=1}^N \alpha_{ij} \phi_{\ell_j} dx }{h},
\quad
i=1, 2, \ldots, g
$$
where $\alpha_{ij}$ is a certain complex number for $i = 1, \ldots, g$
and $j=1, \ldots, N$.
\end{lemma}

\begin{proof}
The Weierstrass curve has the natural projection $\varpi_x : X \to \PP$
and Nagata's Jacobian criterion shows that $dx$ is the natural
one-form. Thus each $\hnuI_i$ has an expression 
$\displaystyle{
 \hnuI_i = \frac{g_i dx }{h_i}
}$ for certain $g_i, h_i \in R$. The $h$ is the  least common multiple of 
$h_i$'s. Due to Riemann-Roch theorem, the numerator must have the finite
order of the singularity at $\infty$ and thus there exists $N<\infty$.
\end{proof}

For the set $\{(h, \hS_h)\}$ whose element satisfies the condition
in Lemma \ref{lem:nuIc}, 
we employ the pair $(h, \hS_h)$ whose $h$ has the least weight
and fix it from here.
The Riemann-Roch theorem enables us to find the basis of 
$H^0(X, \cA_X(*\infty))$.

\begin{definition}
Let us define the subset
$\hS_R$ of $R$ whose element is given by
a finite sum of $\phi_j\in S_R$ with coefficients
$a_{ij} \in \CC$,
$\displaystyle{
\hphi_i := \sum_j a_{ij} \phi_j}$,
satisfying the following relations,
\begin{enumerate}
\item 
for the maximum $j$ in non-vanishing $a_{ij}$, $a_{ij}=1$,
\item
$\wt \hphi_i < \wt \hphi_j$ for $i < j$, and 
\item 
$\displaystyle{H^0(X, \cA_X(*\infty)) = \bigoplus_{i=0} \CC\hphi_i\frac{dx}{h}}$.
\end{enumerate}
\end{definition}

The Riemann-Roch theorem
mean that 
$\displaystyle{
\left\{\nuI_i := \frac{\hphi_{i-1} dx }{h}
\right\}_{i = 1, \ldots, g}}$ is another basis of 
$H^0(X, \cA_X)$.
Let $\hS^{(g)}_R:=\{\hphi_i\}_{i=0, 1, \ldots, g-1}$ which 
gives the canonical embedding of the curve and we call
$\nuI_i$ $(i=1,\ldots,g)$ unnormalized holomorphic one forms.

\begin{definition}
\begin{enumerate}

\item 
The matrix $\omega'$ is defined by
$$
   \nuI = 2\omega' \hnuI, 
\quad \mbox{for}\quad \nuI:=\begin{pmatrix} 
\nu_1 \\ \nu_2 \\ \vdots\\ \nu_g \end{pmatrix},
\quad \hnuI:=\begin{pmatrix} 
\hnuI_1 \\ \hnuI_2 \\ \vdots\\ \hnuI_g \end{pmatrix}.
$$

\item 
The unnormalized Abelian integral and the unnormalized
Abel-Jacobi map are defined by
$$
   \tw(P_1, \ldots, P_k) = 2\omega' \twc(P_1, \ldots, P_k), \quad 
   w(P_1, \ldots, P_k) = 2\omega' \wc(P_1, \ldots, P_k). \quad 
$$
\end{enumerate}
\end{definition}

From the definition, the following lemma is obvious:

\begin{lemma}
By letting $\omega'' := \omega' \tau$, we have
$$
\left(\int_{\alpha_i} \nuI_j\right) = 2 \omega',\quad
\left(\int_{\beta_i} \nuI_j\right) = 2 \omega''.\quad
$$
\end{lemma}

We introduce the unnormalized lattice 
$\Gamma := \langle\omega', \omega''\rangle_{\ZZ}$
 and Jacobian $\JJ(X):=\CC^g/\Gamma$.

\hskip -1.0cm

{\tiny{
\begin{table}[htb]
\caption{ The $\hphi$'s in $\hS^{(g)}_R$
 of Examples \S \ref{rmk:WNF} }
  \begin{tabular}{|r|cccccccccccc}
\hline
 &0&1&2&3&4&5&6&7&8&9&10&11\\
\hline
 I&$1$&-&-&-&-&$x$&-&$y$&-&-&$x^2$&-\\
II&-&-&-&-&-&-&-&$y$&-&-&-&$w$\\
 III&-&-&-&-&-&-&-&$y$&$w$&-&$xy$&$xw$\\
\hline
  \end{tabular}
$\qquad\qquad$
$\qquad\qquad$
$\qquad\qquad$
\begin{tabular}{ccccccccccc|}
\hline
12&13&14&15&16&17&18&19&20&21&22\\
\hline
$xy$&-&$y^2$&$x^3$&-&$x^2y$&-&$xy^2$&$x^4$&$y^3$&$x^3y$\\
$xy$&-&$y^2$&-&$xw$&$x^2y$&$wy$&$xy^2$&-&-&-\\
-&-&-&-&-&-&-&-&-&-&-\\
\hline
  \end{tabular}
\end{table}
}}

Then we have the following lemma:
\begin{lemma}\label{lm:4.3}
\begin{enumerate}
\item The matrix $\omega'$ is regular.
\item The divisor of $\displaystyle{\frac{dx}{h}}$ can be expressed by
$(2g-2+d_1)\infty-\fB$ for a certain effective divisor $\fB$ whose degree
is $d_1 \ge 0$, i.e., 
$\displaystyle{\left(\frac{dx}{h}\right) =
(2g-2+d_1)\infty-\fB}$.
\item $\wt \hphi_{g-1}$ is $(2g-2)+d_1$.
\item $(\hphi_i)\ge(\fB-(2g-2+d_1)\infty)$ for every $\hphi_i 
\in \hS_R^{(g)}$, $(i=0, 1, 2,\ldots,g-1)$.
\item $(\hphi_i)\ge(\fB-(g-1+d_1+i)\infty)$ for every $\hphi_i 
\in \hS_R$, $(i\ge g)$.
\end{enumerate}
\end{lemma}

\begin{proof}
(1) is trivial. (2) is obvious because degree of meromorphic
one-form is $2g-2$ and $h$ is the element of $R$.
From the Riemann-Roch theorem,
$\wt\nuI_g = 0$ and thus, we obtain (3), (4) and (5).
\end{proof}

\begin{remark}
{\rm{
The degree of $\fB$ vanishes, if and only if
the Weierstrass semigroup is symmetric and $2g-1$ is a gap.
}}
\end{remark}

Now we introduce meromorphic functions over $\cS^n X$,
which can be regarded as
a natural generalization of $U$-function of Mumford triplet $UVW$
\cite{Mu}.

\begin{definition} \label{def:mu}
 We define the Frobenius-Stickelberger (FS) matrix with entries in 
$\hS_R$: let $n$ be a positive integer  and
$P_1, \ldots, P_n$ $(1 \le n)$ be in $\Xrs\backslash\infty$,
$$
\hPsi_{n}(P_1, P_2, \ldots, P_n) :=
\begin{pmatrix}
\hphi_0(P_1) &\hphi_1(P_1) & \hphi_2(P_1)  
& \cdots & \hphi_{n-1}(P_1) \\
\hphi_0(P_2) & \hphi_1(P_2) & \hphi_2(P_2)
 & \cdots & \hphi_{n-1}(P_2) \\
\vdots & \vdots & \vdots & \ddots& \vdots \\
\hphi_0(P_n) & \hphi_1(P_{n}) & \hphi_2(P_{n}) 
 & \cdots&  \hphi_{n-1}(P_{n})
\end{pmatrix}.
$$
The
{\it{Frobenius-Stickelberger (FS) determinant}} is
\begin{gather*}
\hpsig{n}(P_1, \ldots, P_n)
 := \det(\hPsig{n}(P_1, \ldots, P_n)).
\end{gather*}
We define the meromorphic function,
$$
\hmu{n}(P): =
\hmu{n}(P; P_1,  \ldots, P_n): =
\lim_{P_i' \to P_i}\frac{1}{\hpsig{n}(P_1' , \ldots, P_n' )}
\hpsig{n+1}(P_1' , \ldots, P_n' , P),
$$
where the $P_i^\prime$ are generic,
the limit is taken (irrespective of the order) for each $i$,
and the meromorphic functions $\hmu{n, k}$'s by
$$
\hmu{n}(P)
 = \phiH{n}(P) +
\sum_{k=0}^{n-1} (-1)^{n-k}\hmu{n, k}(P_1, \ldots, P_n) \phiH{k}(P),
$$
with the convention $\mu_{n, n}(P_1, \ldots, P_n) \equiv $
$\hmu{n, n}(P_1, \ldots, P_n) \equiv 1$.
\end{definition}

\begin{remark}
{\rm{
When $X$ is a hyperelliptic curve, by letting $P\in X$ and $P_i \in X$
 expressed by $(x,y)$ and $(x_i, y_i)$, $\hmu{k}$ is identical to
$U$ in \cite{Mu}:
$$
\hmu{k}(P; P_1, \ldots, P_k) = (x - x_1) (x - x_2) \cdots (x - x_k) 
$$
and each $\hmu{k,i}$ is an elementary symmetric polynomial of 
$x_i$'s. 
}}
\end{remark}

We mention the behavior of the meromorphic function
$\mu_{n, k}(P_1, \ldots, P_n)$,
which is obvious:
\begin{lemma}\label{lem:mu_sing}
For $k<n$,
the order of singularity $\mu_{n, k}(P_1, \ldots, P_n)$ as a
function of $P_n$ at $\infty$ is 
$\wt\hphi_{n}-\wt\hphi_{n-1}$.
\end{lemma}

Using $\mu_{n}(P: P_1, \ldots, P_n)$, we have an addition
structure in Jacobin $\JJ(X)$ as the linear system;
{\lq\lq}$\sim${\rq\rq} means the linear equivalence.

\begin{proposition}
The divisor of $\hmu{g-1}(P; P_1,\ldots,P_{g-1})$ with
respect to $P$ is given by
$$
(\hmu{g-1}) =
\sum_{i=1}^{g-1} P_i +
\sum_{i=1}^{g-1} Q_i + \fB - (2g-2 +d_1) \infty,
$$
where $Q_i$'s are certain points  of $X$.
\end{proposition}

\begin{proof}
It is obvious that there is a non-negative number $N$ 
satisfying the relation,
$$
(\hmu{g-1}) =
\sum_{i=1}^{g-1} P_i +
\sum_{i=1}^N Q_i + \fB - (g-1 +N +d_1) \infty,
$$
and Lemma \ref{lm:4.3} 
(3)
means $N=g-1$.
\end{proof}

\begin{lemma}\label{lem:fB0}
There is a certain divisor $\fB_0$ such that 
$\fB_1 - d_1 \infty = 2 (\fB_0 - d_0 \infty)$ and thus
we have the following linear equivalence,
$$
\sum_{i=1}^{g-1} P_i +\fB_0 - 2(g-1 +d_0) \infty
\sim
-\left(\sum_{i=1}^{g-1} Q_i + \fB_0 - 2(g-1 +d_0) \infty
	\right).
$$
\end{lemma}

\begin{proof}
The Abel-Jacobi theorem implies that since the Abel-Jacobi map is surjective,
the existence $\fB_0$ is obvious.
As the linear equivalence, we have
$$
\sum_{i=1}^{g-1} P_i +
\sum_{i=1}^{g-1} Q_i + 2\fB_0 - 2(g-1 +d_0) \infty
\sim 0,
$$
where $Q_i$'s are certain points  of $X$. 

\end{proof}

Lemma \ref{lem:fB0}
 implies that the image of the Abel-Jacobi map has the symmetry of
the Jacobian  as in the following lemma:

\begin{lemma}\label{lem:Wg-1}
By defining the shifted Abel-Jacobi map,
$$
 w_s (P_1, \ldots,P_{k})
:= w (P_1, \ldots,P_{k}) + w (\fB_0),
$$
the following relation holds:
$$
-w_s (\cS^{g-1}X) = w_s (\cS^{g-1}X).
$$
\end{lemma}

\begin{proof}
Lemma \ref{lem:fB0} means
$ w_s (P_1, \ldots,P_{g-1})
=-w_s (Q_1, \ldots, Q_{g-1})
$
and thus \break
$-w_s (\cS^{g-1}X)$ $ \subset w_s (\cS^{g-1}X)$.
It leads the relation.
\end{proof}

\begin{proposition}\label{prop:cKX}
The canonical divisor $\cK_X \sim (2g-2+2d_0)\infty - 2\fB_0$.
\end{proposition}

For $(\gamma_1,\gamma_2,\ldots,\gamma_k)\in S^k\tX$,
we also define the shifted Abelian integral $w_s$ by
$$
\tw_s(\gamma_1,\gamma_2,\ldots,\gamma_k) := 
\tw(\gamma_1,\gamma_2,\ldots,\gamma_k) +\tw\circ\iota(\fB_0).
$$
We recall our previous results on the Riemann constant \cite{KMP16}:
\begin{proposition}\label{prop:SRC}

\begin{enumerate}
\item
If $d_1>0$, the Riemann constant $\xi$ is not a half period of $\Gamma^\circ$.

\item
The shifted Riemann constant 
$\xi_s:=\xi-\frac{1}{2}\omega^{\prime-1}
 \tw\circ\iota(\fB_0)$ is the half period of $\Gamma^\circ$.

\item
By using the shifted Abel-Jacobi map, we have
$$
		\Theta = \frac{1}{2}\omega^{\prime-1}w_s(\cS^{g-1} X)+\xi_s
\quad \mbox{modulo}\quad \Gamma^\circ,
$$
i.e., for $P_i \in X$,
$
\theta\left(\frac{1}{2}\omega^{\prime-1}\tw_s\circ\iota(P_1, \ldots, P_{g-1})
+\xi_s\right)=0$.

\item
There is a $\theta$-characteristic $\delta$ of a half period which 
represents
the shifted Riemann constant $\xi_s$, i.e.,
$
\theta[\delta]
\left(\frac{1}{2}\omega^{\prime-1}\tw_s\circ\iota(P_1, \ldots, P_{g-1}) \right)=0.
$

\end{enumerate}

\end{proposition}

In order to describe the Jacobi inversion formulae for the strata of
the Jacobian, we define the Wirtinger variety $\WW_k:=w_s(\cS^k X)$
and 
$$
\Theta_k:= \WW_k \bigcup [-1] \WW_k,
$$
where $[-1]$ is the minus operation on the Jacobian $\JJ(X):=\Theta_g$.
We also define the strata:
$\WW_{s,1}^k:=w_s(\cS^k_1 X)$
($\WW_1^k:=(\cS^k_1 X)$),
where 
$
\cS^n_m(X) := \{D \in \cS^n(X) \ | \
    \mathrm{dim} | D | \ge m\}.
$

The Abel-Jacobi theorem and Lemma \ref{lem:Wg-1} mean
that $\WW_g = \JJ(X)$ and $[-1]\WW_{i}=\WW_i$ for $i = g-1, g$ but
in general, $[-1]\WW_{i}$ does not equal to $\WW_i$ for $i<g-1$.

\section{Jacobi inversion formulae for any Weierstrass curve}

The $\mu$ function enables us to rewrite 
Jorgenson's relation in Proposition \ref{prop:Jo}.
Since any element $u=^t(u_1, u_2, \ldots, u_g)$ of $\CC^g$ is given as 
$u = \tw(\gamma_1, \ldots, \gamma_g)$ for a certain element
$(\gamma_1, \ldots, \gamma_g) \in S^g\tX$,
we consider the differential $\displaystyle{
\partial_i :=\frac{\partial}{\partial u}}$ of a function on
$\CC^g$ and  
have the following proposition:

\begin{proposition}
For $P_1, P_2, \ldots, P_{\ell} \in X$ and $\ell=g-1$,
we have
$$
\hmu{g}(P; P_1, \ldots, P_{\ell})
=
\hphi_{\ell}(P)
+\sum_{i=1}^g 
\frac{
\displaystyle{
\partial_i 
\theta\left(\frac{1}{2}\omega^{\prime-1} 
\tw\circ\iota(P_1, P_2, \cdots, P_{\ell})+\xi\right)
}
}{
\displaystyle{
\partial_{\ell} 
\theta\left(\frac{1}{2}\omega^{\prime-1} 
\tw\circ\iota(P_1, P_2, \cdots, P_{\ell})+\xi\right)
}
}
\hphi_{i-1}(P).
$$
The right hand side does not depend on the choice of the
embedding $\iota$.
\end{proposition}

\begin{proof}
First we rewrite Proposition \ref{prop:Jo} in terms of
the unnormalized holomorphic one forms.
It is obvious that
for general complex numbers
$a_i$ and $b_i$ $(i=1,2, \ldots, g)$, we have 
$$
\frac{
\left|\begin{matrix}
\nuI_1(P_1) & \nuI_2(P_1) & \cdots &\nuI_g(P_1)\\
\nuI_1(P_2) & \nuI_2(P_2) & \cdots &\nuI_g(P_2)\\
\vdots & \vdots & \ddots & \vdots \\
\nuI_1(P_{\ell}) & \nuI_2(P_{\ell}) & \cdots &\nuI_g(P_{\ell})\\
a_1 & a_2 & \cdots & a_g \\
\end{matrix}
\right|
}{
\left|\begin{matrix}
\nuI_1(P_1) & \nuI_2(P_1) & \cdots &\nuI_g(P_1)\\
\nuI_1(P_2) & \nuI_2(P_2) & \cdots &\nuI_g(P_2)\\
\vdots & \vdots & \ddots & \vdots \\
\nuI_1(P_{\ell}) & \nuI_2(P_{\ell}) & \cdots &\nuI_g(P_{\ell})\\
b_1 & b_2 & \cdots & b_g \\
\end{matrix}
\right|
}
=
\frac{
\displaystyle{
\sum_{i=1}^g a_i 
\partial_i
	\theta(\frac{1}{2}\omega^{\prime -1}\tw\circ\iota(
        P_1, P_2, \cdots, P_{\ell})+\xi)
}
}{
\displaystyle{
\sum_{i=1}^g b_i
\partial_i
	\theta(\frac{1}{2}\omega^{\prime -1}\tw\circ\iota(
        P_1, P_2, \cdots, P_{\ell})+\xi)
}
}.
$$
By letting  $(a_1, \ldots, a_g)=(\nuI_1(P), \ldots, \nuI_g(P))$
and  $(b_1, \ldots, b_g)=(0, \ldots, 0, \nuI_g(P))$ for
a generic point $P \in X$,
the relation is reduced to the proposition.
\end{proof}

Using Jorgenson's result with the ordering of $\hphi_i$, 
we show our main theorem:

\begin{theorem}\label{thm:JIF}
For $(P_1, \ldots, P_k) \in \cS^k X\setminus \cS^k_1 X$ $(k<g)$ 
and a positive integer
$i\le k$, 
we have the relation, 
$$
\hmu{k,i-1}(P_1,\ldots, P_k) = 
\frac{
\displaystyle{
\partial_i 
\theta\left(\frac{1}{2}\omega^{\prime-1} \tw\circ\iota(P_1, P_2, \cdots, P_{k})+\xi\right)
}
}{
\displaystyle{
\partial_{k+1} 
\theta\left(\frac{1}{2}\omega^{\prime-1} \tw\circ\iota(P_1, P_2, \cdots, P_{k})+\xi\right)
}
},
$$
and for a certain $\theta$-characteristic $\delta$ for $\xi_s$,
$$
\hmu{k,i-1}(P_1,\ldots, P_k) = 
\frac{
\displaystyle{
\partial_i 
\theta[\delta]\left(\frac{1}{2}\omega^{\prime-1} 
                 \tw_s\circ\iota(P_1, P_2, \cdots, P_{k})\right)
}
}{
\displaystyle{
\partial_{k+1} 
\theta[\delta]\left(\frac{1}{2}\omega^{\prime-1} 
                  \tw_s\circ\iota(P_1, P_2, \cdots, P_{k})\right)
}
},
$$
especially $k=1$ case,
$$
\frac{\hphi_1(P_1)}{\hphi_0(P_1)}
=
\frac{
\displaystyle{
\partial_1 
\theta[\delta]\left(\frac{1}{2}\omega^{\prime-1} \tw_s\circ\iota(P_1)\right)
}
}{
\displaystyle{
\partial_{2} 
\theta[\delta]\left(\frac{1}{2}\omega^{\prime-1} \tw_s\circ\iota(P_1)\right) }}.
$$
\end{theorem}

\begin{proof}
Since these $\hphi_i$'s have the natural ordering for $P_k \to \infty$,
we could apply the proof of Theorem 5.1 (3) in \cite{MP08} to this case.
It means that we prove it inductively.
For $k<g$ and $i<k$, 
using the property of Lemma \ref{lem:mu_sing}, we consider 
$\hmu{k,i-1}(P_1, P_2, \ldots, P_k)/\hmu{k,k-1}(P_1, P_2, \ldots, P_k)$
and its limit $P_k \to \infty$. Then we have the relation for
$\hmu{k-1,i-1}$. 
\end{proof}

Theorem \ref{thm:JIF} means the Jacobi inversion formula of $\Theta_k$
for a general Weierstrass curve.
Using the natural basis of the Weierstrass non-gap sequence
 of the Weierstrass curve, we can consider the precise structure of
strata $\Theta_k$ of the Jacobian.
The relation contains the cyclic $(r,s)$ curve cases
and trigonal curve cases reported in \cite{MP08,MP14,KMP18}. 

Since for every compact Riemann surface $Y$,
there exists a Weierstrass curve $(X, \infty)$
which is birationally-equivalent
to $Y$, 
Theorem \ref{thm:JIF} gives the structure of the strata in the Jacobian
of $Y$.

\begin{remark}
{\rm{
Nakayashiki investigated precise structure of the $\theta$ function
for a pointed compact Riemann surface whose Weierstrass gap is associated
with a Young diagram \cite{N16}. He refined the Riemann-Kempf
theory \cite{ACGH} in terms of the unnormalized holomorphic one-forms.
Using the results, it is easy to rewrite the main theorem using
non-vanishing quantities as in \cite{MP14,KMP18}.
}}
\end{remark}

\begin{remark}
{\rm{
Though it is well-known that the sine-Gordon equation is an algebraic relations
of the differentials of $\mu_g$-function in the hyperelliptic Jacobian $\JJ(X)$
\cite{Mu,Pr86}, 
the equation can be generalized 
to one in strata in the Jacobian
$\JJ(X)$ using the $\mu_k$ function $k<g$
\cite{M05}.
Ayano and Buchstaber investigated similar structure of stratum of
the hyperelliptic Jacobian of genus three to find a novel
differential equation which characterizes the stratum \cite{AB}.
In other words, it is expected that these relations in our theorem might show
some algebraic relations in the strata in the Jacobian of a general
compact Riemann surface.

For example,
for a Weierstrass curve $X$ whose $\hphi_1/\hphi_0$ equals to $x$,
we have the Burgers equation,
$$
             \frac{\partial}{\partial u_i} 
\frac{ \partial_1\theta[\delta]\left(
              \frac{1}{2}\omega^{\prime-1} w_s(P)\right) }
{ \partial_2\theta[\delta]\left(\frac{1}{2}\omega^{\prime-1} w_s(P)\right) }
-
\frac{\hphi_j(P)}{\hphi_i(P)}
             \frac{\partial}{\partial u_j} 
\frac{ \partial_1\theta[\delta]
        \left(\frac{1}{2}\omega^{\prime-1} w_s(P)\right) }
{ \partial_2\theta[\delta]\left(\frac{1}{2}\omega^{\prime-1} w_s(P)\right) }=0.
$$
}}
\end{remark}

\begin{remark}
{\rm{
Korotkin and Shramchenko defined the $\sigma$ function for a general
 compact Riemann surface based on the Klein's investigation \cite{KS};
they assume that the $\theta$ characteristic $\delta$ as 
a half period of Jacobian $\Gamma^o$ 
is given by the Riemann constant $\xi$ with $w(D_s)$ of the divisor
$D_s$ of the spin structure of the curve $X$.
The spin structure corresponds to $\fB_0$ 
in  Propositions \ref{prop:SRC} and \ref{prop:cKX} \cite{At}.

We have constructed the $\sigma$ functions for trigonal curves 
\cite{MP08,MP14,KMP18} using the EEL-construction proposed by 
Eilbeck, Enolski\u{\i} and  Le\u{\i}kin \cite{EEL} 
whose origin is Baker \cite{B97}; the EEL-construction
differs from Klein's approach.
Due to the arguments on this paper,
Emma Previato posed a problem whether 
we could reproduce the $\sigma$ function of Korotkin and Shramchenko 
by means of the EEL-construction for the Weierstrass curve.
It means an extension of the EEL-construction of the $\sigma$ function.

Further using the extension and Nakayashiki's investigation in \cite{N16}, 
we could connect our results tr
Sato Grassmannian in the Sato universal Grassmannian theory
and investigate the $\sigma$ function of a Weierstrass curve more precisely.

Since it is known that there are infinitely many numerical semigroups which are
not Weierstrass, but the distribution of Weierstrass semigroups in the
set of the numerical semigroups is not clarified.
Thus it is a natural question how the Weierstrass curves are characterized in
Sato Grassmannian manifolds.
}}
\end{remark}

\section{An Example}
\subsection{Curve $X$ of Example II
 in \S \ref{rmk:WNF}}
Let us consider the curve $X$ of
 Example II in \S \ref{rmk:WNF}.
For a point $P_i=(x_i, y_i, w_i)$ $(i=1,2,\ldots,7)$ of $X$,
the $\hmu{7}$ is given by
$$
\hmu{7}(P,P_1,P_2,\ldots,P_7)=
\frac{
\left|\begin{matrix}
y_1& w_1& x_1y_1& y_1^2 & x_1 w_1& x_1^2y_1& w_1y_1 & x_1y_1^2\\
y_2& w_2& x_2y_2& y_2^2 & x_2 w_2& x_2^2y_2& w_2y_2 & x_2y_2^2\\
y_3& w_3& x_3y_3& y_3^2 & x_3 w_3& x_3^2y_3& w_3y_3 & x_3y_3^2\\
y_4& w_4& x_4y_4& y_4^2 & x_4 w_4& x_4^2y_4& w_4y_4 & x_4y_4^2\\
y_5& w_5& x_5y_5& y_5^2 & x_5 w_5& x_5^2y_5& w_5y_5 & x_5y_5^2\\
y_6& w_6& x_6y_6& y_6^2 & x_6 w_6& x_6^2y_6& w_6y_6 & x_6y_6^2\\
y_7& w_7& x_7y_7& y_7^2 & x_7 w_7& x_7^2y_7& w_7y_7 & x_7y_7^2\\
y& w& xy& y^2 & x w& x^2y& wy & xy^2\\
\end{matrix}\right|}
{\left|\begin{matrix}
y_1& w_1& x_1y_1& y_1^2 & x_1 w_1& x_1^2y_1& w_1y_1 \\
y_2& w_2& x_2y_2& y_2^2 & x_2 w_2& x_2^2y_2& w_2y_2 \\
y_3& w_3& x_3y_3& y_3^2 & x_3 w_3& x_3^2y_3& w_3y_3 \\
y_4& w_4& x_4y_4& y_4^2 & x_4 w_4& x_4^2y_4& w_4y_4 \\
y_5& w_5& x_5y_5& y_5^2 & x_5 w_5& x_5^2y_5& w_5y_5 \\
y_6& w_6& x_6y_6& y_6^2 & x_6 w_6& x_6^2y_6& w_6y_6 \\
y_7& w_7& x_7y_7& y_7^2 & x_7 w_7& x_7^2y_7& w_7y_7 \\
\end{matrix}\right|}.
$$
Thus we have the following Proposition:
\begin{proposition} \label{prop:4.5}

\noindent
1) For $(P, P_1, P_2,\ldots, P_k) \in X \times
	\left(\cS^k(X) \setminus \cS^k_1(X)\right)$ for $k<8$, 
we have
$$
\hmu{k,i-1}(P_1,\ldots, P_k) = 
\frac{
\displaystyle{
\partial_i 
\theta[\delta]\left(\frac{1}{2}\omega^{\prime-1} w_s(P_1, P_2, \cdots, P_{k})
               \right)
}
}{
\displaystyle{
\partial_{8} 
\theta[\delta]\left(\frac{1}{2}\omega^{\prime-1} w_s(P_1, P_2, \cdots, P_{k})
              \right)
}
}
$$
especially
$$
\frac{
\displaystyle{
\partial_1 
	\theta[\delta]\left(\frac{1}{2}\omega^{\prime-1} w_s(P_1)\right)
}
}
{
\displaystyle{
\partial_2 
	\theta[\delta]\left(\frac{1}{2}\omega^{\prime-1} w_s(P_1)\right)}
}
	= \frac{w_1}{y_1}.
$$
\end{proposition}


\end{document}